\documentclass[11pt]{article}   
\usepackage[dvips]{graphicx}
\usepackage{amsmath,amssymb}
\usepackage{algorithm}
\usepackage{algorithmic}

\date{}

\newtheorem{theorem}{Theorem}
\newtheorem{lemma}{Lemma}

\newtheorem{proposition}{Proposition}

\date{}

\newenvironment{proof}[1][\hspace{-1.0ex}]%
{\par\addvspace{1mm}{\it Proof\hspace{1.0ex}{#1}.} }%
{\quad$\blacktriangle$\par\addvspace{1mm}}

    \newif\ifNoRemark
    \def\addtheorem#1#2#3#4{ 
    \ifthenelse{\expandafter\isundefined\csname the#2\endcsname}{\newcounter{#2}}{}
    \newenvironment{#1}[1][\global\NoRemarktrue]
     {\par\addvspace{2mm}\noindent 
       \refstepcounter{#2}{\bf #3~\csname the#2\endcsname
      \vphantom{##1}\ifNoRemark.\ \else\ (##1).\fi}\begingroup #4}%
     {\endgroup\par\addvspace{1mm}\global\NoRemarkfalse}
    \expandafter\newcommand\csname b#1\endcsname{\begin{#1}}
    \expandafter\newcommand\csname e#1\endcsname{\end{#1}}
    }

\begin{document}

\title{Constructions of Pairs of Orthogonal Latin Cubes\footnote{\,The research was
carried out at the Sobolev Institute of Mathematics at the expense
of the Russian Science Foundation 18-11-00136.}}

\author{ Vladimir N. Potapov\\
Sobolev Institute of Mathematics, Novosibirsk, Russia \\
vpotapov@math.nsc.ru\\}

\maketitle

\begin{abstract}
A pair of orthogonal latin cubes of order $q$ is equivalent to an
MDS code with distance $3$ or to an ${\rm OA}_1(3,5,q)$  orthogonal
array. We construct pairs of orthogonal latin cubes for a sequence
of previously unknown orders $q_i=16(18i-1)+4$ and
$q'_i=16(18i+5)+4$. The minimal new obtained parameters of
orthogonal arrays are ${\rm OA}_1(3,5,84)$.

\textit{Keywords}--- latin square, latin cube, MOLS, MDS code, block
design, Steiner system, orthogonal array

\end{abstract}

MSC2010: 05B15, 94B05,  05B05

\section{Introduction}

A {\it latin square} of order $q$ is  a $q\times q$ array of $q$
symbols where each symbol occurs exactly once in every row and in
every column.  A $k$-dimensional array satisfying the same condition
is called a {\it latin $k$-cube}. Any $2$-dimensional axis-aligned
plane (face) of a latin $k$-cube of order $q$ is a latin square of
order $q$ by definition. Two latin squares are {\it orthogonal} if,
when they are superimposed, every ordered pair of symbols appears
exactly once. For brevity, a pair of orthogonal latin squares is
called POLS. If in a set of latin squares, any two latin squares are
orthogonal then this set is called a system of Mutually Orthogonal
Latin Squares (MOLS). Two latin $k$-cubes are {\it orthogonal} if
any pair of corresponding $2$-dimensional faces of these cubes is a
POLS. Bose, Shrikhande and Parker \cite{BSP} proved that for each
positive integer $q$, $q\neq 2,6$, there exists  POLS of order $q$
and POLS of orders $2$ and $6$ are not exist. As a corollary we
obtain nonexistence of pairs of orthogonal latin $k$-cubes of orders
$2,6$. A nonexistence of pairs of orthogonal latin $k$-cubes of
orders $q$ if $k>q-1$ follows from the sphere-packing (Hamming)
bound. But the complete spectrum of possible orders of pairs of
orthogonal latin $k$-cubes remains unknown for any $k\geq 3$. Ten is
the minimum order for which it is not known whether  pairs of
orthogonal latin $3$-cubes exist. In this paper we construct pairs
of orthogonal latin $3$-cubes for a sequence  of previously unknown
orders $q_i=16(18i-1)+4$ and $q'_i=16(18i+5)+4$. New pairs of
orthogonal latin $3$-cubes are created as files. These files are
available on the website
https://ieee-dataport.org/open-access/graeco-latin-cubes. The
minimum new order for a pair of orthogonal latin cubes obtained by
proposed construction is $84$.

Let $Q_q=\{0,\dots,q-1\}$. A subset $C$ of $Q_q^d$ is called an
$MDS(t,d,q)$ {\it code} (of order $q$, code distance $t+1$ and
length $d$) if $|C\cap\Gamma|=1$ for each $t$-dimensional
axis-aligned plane $\Gamma$. Ethier and Mullen \cite{EMullen} proved
that $MDS(2,2+s,q)$ codes are equivalent to pairs of orthogonal
latin $s$-cubes of order $q$. There are two well-known methods for
constructing  MDS codes. If $q$ is a prime power, then we can
consider  $Q_q$ as the Galois field $GF(q)$. MDS codes obtained as
the solution of an appropriate system of linear equations over
$GF(q)$ are known as Reed--Solomon codes. If there exists an
$MDS(t,d,p_1)$ code and an $MDS(t,d,p_2)$ code, then we get an
$MDS(t,d,p_1p_2)$ code by a product construction (McNeish's
theorem). We represent a new construction of $MDS(2,5,q)$ codes that
is similar to  Wilson's construction  for pairs of orthogonal latin
squares with aligned subsquares (see, \cite{DK} and \cite{Zhu}).

The problem of existence of MDS codes with non-prime-power orders is
 connected to the problem of existence of Steiner block designs. By
methods of random graph theory Keevash \cite{Keevash1} and Glock et
al. \cite{Glock} proved that the natural divisibility conditions are
sufficient for existence of Steiner system $S(t,k,n)$ apart from a
finite number of exceptional $n$ for given fixed $t$ and $k$. It is
not difficult to prove that any MDS code is equivalent to a
transversal in an appropriate multipartite hypergraph (see
\cite{Pot}). Then the existence of  $MDS(t,d,q)$ codes   follows
from   \cite{Keevash3} (Theorem 1.7) apart from a finite number of
exceptional $q$ for given fixed $t$ and $d$. In the last section of
this paper we propose a construction of pairs of orthogonal latin
$3$-cubes based on Steiner block designs.

Note that an $MDS(2,q+1,q)$ code  (a pair of orthogonal
$(q-1)$-cubes) is an  $1$-error correcting perfect code. The
existence of such codes is a well-known problem if $q$ is not a
prime power (see \cite{MS}).

\section{Connection between MDS codes and orthogonal systems}

An ${\rm OA}_\lambda (s,d,q)$  orthogonal array is a $\lambda
q^s\times d$ array whose entries are from $Q_q$ such that in every
subset of $s$ columns of the array, every $s$-tuple from $Q_q^s$
appears in exactly $\lambda$ rows. Further we consider only
orthogonal arrays with $\lambda=1$. In this case every column of the
orthogonal array is a function $f:Q_q^s\rightarrow Q_q$. A set of
columns of an orthogonal array with $\lambda=1$ is called an
orthogonal system. In other words, a system consisting of $d$
functions $f_1,\dots,f_d$, $f_i:Q_q^s\rightarrow Q_q$ ($d\geq s$) is
{\it orthogonal} if for each subsystem $f_{i_1},\dots,f_{i_s}$
consisting of  $s$ functions it holds
$$ \{(f_{i_1}(\overline{x}),\dots,f_{i_s}(\overline{x})) \ |\ \overline{x}\in Q_q^{s}\}=Q_q^{s}.$$
If the system remains orthogonal after substitution any constants
for each subset of variables, then it is called {\it
strong-orthogonal}. If the number of variables is two, then such
system is a system of MOLS (see \cite{EMullen}). If $s=3$, it is a
set of Mutually Orthogonal Latin Cubes (MOLC). Ethier and Mullen
\cite{EMullen} proved that MDS codes are equivalent to
strong-orthogonal systems. Moreover, by a replacement of variables
it is possible to obtain a strong-orthogonal system consisting of
$d-s$ functions from  any orthogonal system consisting of $d$
functions over $Q_q^s$.

\
\begin{proposition}\label{proEM}
The following conditions are equivalent:\\
1) a system consisting of $t$ functions $f_1,\dots,f_t$,
$f_i:Q_q^s\rightarrow Q_q$  is { strong-orthogonal};\\
2) the set $C=\{(x_1,\dots,x_s,f_1(\overline{x}),\dots
f_t(\overline{x})) : x_i\in Q_q\}$ is an $MDS(t,t+s,q)$ code; \\
3) the array consisting of all elements of $C$ as rows is an ${\rm
OA}_1 (s,t+s,q)$ orthogonal array.
\end{proposition}

A projection (punctured code) of an $MDS(t,t+s,q)$ code onto a
hyperplane  is equal to a removal of one of the functions $f_i$. The
punctured code  is an $MDS(t-1,t+s-1,q)$ code by Proposition
\ref{proEM}. Consequently, an existence of $MDS(2,2+s,q)$ code or a
pair of orthogonal latin $s$-cubes of order $q$  follows from an
existence of  an $MDS(t,t+s,q)$ code if  $t\geq 2$.

Sometimes the terms ``latin cube" and ``$t$ mutually orthogonal
latin cubes" is used for  ${\rm OA}_q (2,4,q)$   and ${\rm OA}_q
(2,t+3,q)$  orthogonal arrays respectively. It is easy to see that
our definition of a system of MOLC is stronger.

\section{Constructions of MDS codes}

The Hamming distance $\rho$ between two elements of $Q_q^d$ is the
number of positions at which the corresponding symbols are
different. In this paper we use only the Hamming distance. The code
distance of $C\subset Q_q^d$ is $\rho(C)=\min\limits_{x\in C,y\in
C,x\neq y}\rho(x,y)$. The distance between two subsets $A,B\subset
Q_q^d$ is $\min\limits_{x\in A,y\in B}\rho(x,y)$. The Singleton
bound  for the cardinality of a code $C\subset Q_q^d$ with  distance
$t+1$  is $|C|\leq q^{d-t}$.
 MDS codes  achieve equality in this
 bound.

\begin{proposition}\label{proNM5}
A subset $C\subset Q_q^d$ with code distance $t+1$ is an {\rm MDS}
code if and only if $|C|=q^{d-t}$.
\end{proposition}

The Hamming bound for the cardinality of code $C\subset Q_q^d$ with
distance $3$ is $|C|\leq \frac{q^d}{1+(q-1)d}$. Then the
inequalities $q^{d-2}\leq \frac{q^d}{1+(q-1)d}$ or $d\leq q+1$ are a
necessary condition for the existence of an $MDS(2,d,q)$ code.
Consequently, an $MDS(2,5,3)$ code or a pair of orthogonal latin
cubes of order $3$ do not exist. Moreover, by puncturing codes we
have a necessary condition  $s\leq q-1$ for the existence of an
$MDS(t,t+s,q)$  if $t\geq 2$. For linear codes this condition $s\leq
q-1$ is in \cite{MS}.

Let $q$ be a prime power and let $Q_q=GF(q)$.  A linear
$k$-dimensional subspace $C\subset Q_q^d$ with  distance $t$ is
called $[d,k,t]$ code over $GF(q)$. By Proposition \ref{proNM5} we
see that any $[d,d-t,t+1]$ code over $GF(q)$ is an $MDS(t,d,q)$
code. By using a well-known construction of a linear MDS code (see
\cite{MS}, Chapters 10,11, or \cite{Glock}, Theorem 9.1) by means of
an appropriate parity-check matrix  over $GF(q)$ we can conclude
that the following proposition is true.

\begin{proposition}\label{proNMb}
Let $q$ be a prime power. Then
 for each  integers $d\leq q+1$  and $\varrho$,
$3\leq \varrho < d$, there exists a linear (over $GF(q)$) {\rm MDS}
code $C\subset Q_q^d$ with code distance $\varrho$.
\end{proposition}

We will say that an $MDS(t,d,q)$ code $M_0$ is a {\it super}
$MDS(t,d,q)$ code if there exist $MDS(t+1,d,q)$ code $M_1$ and
$MDS(t+2,d,q)$ code $M_2$ such that $M_2\subset M_1\subset M_0$.

By removal of any row from a parity-check matrix of a linear {\rm
MDS} code with distance $t+1$, we obtain a parity-check matrix of an
{\rm MDS} code with distance $t$ that contains the original code.
Thus Propositions \ref{proNMbb} and \ref{proNMbc} follow from
Proposition \ref{proNMb}.

\begin{proposition}\label{proNMbb}
Let $q$ be a prime power. Then
 for each  integers $d\leq q+1$  and $\varrho$,
$3\leq \varrho < d-2$, there exists a linear over $GF(q)$ super {\rm
MDS} code $C\subset Q_q^d$ with code distance $\varrho$.
\end{proposition}

\begin{proposition}\label{proNMbc}
Let $q$ be a prime power. Then
 for each  integers $d\leq q+1$  and $\varrho$,
$3\leq \varrho < d-1$, there exists a linear over $GF(q)$  {\rm MDS}
code $C\subset Q_q^d$ with code distance $\varrho$ that is an union
of $q$ disjoint linear over $GF(q)$  {\rm MDS} code $C\subset Q_q^d$
with code distance $\varrho+1$.
\end{proposition}

The set $Q_{q_1q_2}$ can be considered as the Cartesian product
$Q_{q_1}\times Q_{q_2}$. Consequently, we can identify
$Q^d_{q_1}\times Q^d_{q_2}$ and the hypercube $Q^d_{q_1q_2}$. Thus
if $C_1\subset Q^d_{q_1}$ and $C_2\subset Q^d_{q_2}$ then
$$C_1\times C_2=\{((x_1,y_1),(x_2,y_2),\dots,(x_d,y_d)) :
(x_1,\dots,x_d)\in C_1, (y_1,\dots,y_d)\in C_2\}\subset
Q^d_{q_1q_2}.$$

\begin{proposition}[McNeish]\label{proNMb1}
Suppose  $M_1$ is an (super) $MDS(t,d,q_1)$ code and $M_2$ is an
(super) $MDS(t,d,q_2)$ code. Then $M_1\times M_2$ is  an (super)
$MDS(t,d,q_1q_2)$ code.
\end{proposition}

By combining results of Propositions \ref{proNMb} and \ref{proNMb1}
we obtain that  $MDS(2,5,q)$ codes  exist if
$q=2^{\delta_2}3^{\delta_3}5^{\delta_5}\dots$, where $\delta_2\neq
1$ and $\delta_3\neq 1$.

Let $A\subset Q_q$. Denote by $\pi_A$ a function mapping from
$Q_p\times Q_q$ to $Q_{p(q-|A|)+|A|}$ by the following rule:
$\pi_A(x,y)=(x,y)$ if $y\not \in A$, and $\pi_A(x,y)=y$ if $y\in A$.
Let $C_1\subset Q^d_{p}$ and $C_2\subset Q^d_{q}$. Denote
$C_1\times_A C_2=\{(\pi_A(z_1),\dots,\pi_A(z_d)) : \overline{z}\in
C_1\times C_2\}$.
For any $C\subset Q^d_q$ we denote by $U_t(C)$  the $t$-neighborhood
of $C$, i.\,e., $U_t(C)=\{x\in Q^d_q : \exists y\in C, \rho(x,y)\leq
t\}$.


A set $D\subset Q^d_q$ is called an $MDS(t,d,q)$ with $j$-$A$-hole
($t+1\leq j\leq d$) if\\ 1) the code distance of $D$ is equal to $t+1$;\\
2) $D\cap U_{d-j+t}(A^d)=\varnothing$;\\ 3)
$U_{d-j+t}(D)=Q^d_q\backslash A^d$;\\ 4)
$|D|=\sum\limits_{k=0}^{j-t-1}{d-t \choose k}(q-|A|)^{d-t-k}|A|^k$.

For $t=2$ and $d=5$ we get that an $MDS(2,5,q)$ code with
$5$-$A$-hole has cardinality $q^3 -|A|^3$ and an $MDS(2,5,q)$ code
with $4$-$A$-hole has cardinality $(q-|A|)^3 +3(q-|A|)^2|A|$.

Suppose that $M$ is an $MDS(t,d,q)$ code, $a\in Q_q$,  and
$\overline{a}=(a,\dots,a)\in M$. It is easy to see that $M\setminus
\{\overline{a}\}$ is an $MDS(t,d,q)$ with $d$-$\{a\}$-hole. Let
$A\subset Q_q$ and let $M$ be an MDS code, $M\subset Q_q^d$. A
subset $M\cap A^d$   is called a {\it subcode}  of  $M$ if it is an
MDS code  in $A^d$ with the same code distance as $M$. If $M\cap
A^d$ is a subcode, then $M\setminus A^d$ is an $MDS(t,d,q)$ with
$d$-$A$-hole.

Let us formulate a known construction of a POLS (see \cite{DK},
Chapter 4) in introduced terms.

\begin{proposition}\label{theoMOLS0}
Suppose that
\begin{itemize} \item $ M_1\subset M$ is
$MDS(3,4,p)$ code and $M$ is $MDS(2,4,p)$ code, \item $D$ is an
$MDS(2,4,q)$ code,
\item $E$ is an $MDS(2,4,q_1-q)$ code on alphabet $A$,\item $F$ is
an $MDS(2,4,q_1)$ code with $4$-$A$-hole, where $|A|=q_1-q$.
\end{itemize}
 Then
 the set $C=E\cup (M_1\times_A F) \cup
((M\setminus M_1)\times D)$ is an $MDS(2,4,(p-1)q+q_1)$ code.
\end{proposition}

Consider an example of code $C$ that is described in Proposition
\ref{theoMOLS0}. An $MDS(2,4,p)$ code is equivalent to a POLS.
Determine $p=q_1=4$ and $q=3$. Let $M$ corresponds to the pair $
   \begin{array}{ccccccccc}
  0 &1 &2 & 3&\quad           & 0 &1 &2 &3  \\
  3 &2 &1 &0&\quad           &  2 &3 &0 &1  \\
  1 &0 &3 &2 &\quad         &3 &2 &1 &0 \\
  2 &3 &0 &1 &\quad           &      1& 0& 3&2 \\
     \end{array} $
and let  $M_1$ corresponds to main diagonals of this squares.
Suppose that $D$ corresponds to the pair $\begin{array}{ccccccc}
  b &c &e &\quad           & b &c &e  \\
  c &e &b &\quad           &  e &b &c  \\
  e &b &c  &\quad         &c &e &b  \\
       \end{array} $, $F$ corresponds to the pair\\  $
   \begin{array}{ccccccccc}
  a &b &c &e&\quad          & b &a &e &c  \\
  c &e &a &b &\quad         &a &b &c &e \\
  b &a &e &c &\quad         & c& e& a&b \\
  e &c &b &. &\quad         & e &c &b &.  \\
     \end{array} $  and $E=(a,a,a,a)$. Then the constructed code $C$ is equivalent to the
     following
     POLS of order $13$:

  $
   \begin{array}{rrrrrrrrrrrrrrrrr}
  a &0b &0c |&1b &1c  &1e|&  2b &2c &2e|  &3b &3c &3e |& 0e \\
    0c &0e &a |&1c &1e&1b|&  2c &2e &2b|  &3c &3e &3b  |& 0b \\
  0b &a &0e |&1e &1b  &1c|&  2e &2b &2c|  &3e &3b &3c  |&  0c\\
 \hline
  3b &3c &3e |& a &2b &2c |&1b &1c  &1e|&  0b &0c  &0e|& 2e\\
  3c &3e &3b |& 2c &2e &a |&1c &1e  &1b|&  0c &0e  &0b|&  2b\\
  3e &3b &3c |& 2b &a &2e |&1e &1b  &1c|&  0e &0b  &0c|& 2c\\
  \hline
  1b &1c  &1e|& 0b &0c  &0e|&a &3b &3c|&  2b &2c &2e|&3e\\
  1c &1e  &1b|& 0c &0e  &0b|&3c &3e &a|&  2c &2e &2b|&3b\\
  1e &1b  &1c|& 0e &0b  &0c|&3b &a &3e|&  2e &2b &2c|&3c\\
  \hline
  2b &2c &2e  |& 3b &3c &3e |&  0b &0c  &0e|& a &1b &1c |& 1e     \\
  2c &2e &2b  |& 3c &3e &3b |&  0c &0e  &0b|& 1c &1e &a |& 1b    \\
  2e &2b &2c  |& 3e &3b &3c |&  0e &0b  &0c|& 1b &a &1e |&  1c\\
\hline
0e &0c &0b |&2e &2c &2b |& 3e &3c &3b |&1e &1b &1c |&   a\\
   \end{array} $

\vskip7mm
   $
    \begin{array}{rrrrrrrrrrrrrrrrr}
   0b &a &0e   |&1b &1c  &1e|&  2b &2c &2e|  &3b &3c &3e |& 0c \\
   a &0b &0c   |&1e &1b&1c  |&  2e &2b &2c|  &3e &3b &3c  |& 0e \\
   0c &0e &a   |&1c &1e  &1b|&  2c &2e &2b|  &3c &3e &3b  |&  0b\\
  \hline
    2b &2c &2e|&  3b &a &3e |&0b &0c  &0e|& 1b &1c  &1e |& 3c\\
    2e &2b &2c|&  a &3b &3c |&0e &0b  &0c|& 1e &1b&1c   |&  3e\\
    2c &2e &2b|&  3c &3e &a |&0c &0e  &0b|& 1c &1e  &1b |& 3b\\
   \hline
   3b &3c &3e|&2b &2c &2e |&1b &a &1e |& 0b &0c  &0e|&1c\\
   3e &3b &3c|&2e &2b &2c |&a &1b &1c |& 0e &0b  &0c|&1e\\
   3c &3e &3b|&2c &2e &2b |&1c &1e &a |& 0c &0e  &0b|&1b\\
   \hline
   1b &1c  &1e|& 0b &0c  &0e|& 3b &3c &3e |&2b &a &2e |& 2c     \\
   1e &1b&1c  |& 0e &0b  &0c|& 3e &3b &3c |&a &2b &2c |& 2e    \\
   1c &1e  &1b|& 0c &0e  &0b|& 3c &3e &3b |&2c &2e &a |&  2b\\
 \hline
 0e &0c &0b |&3e &3c &3b |& 1e &1c &1b |&2e &2b &2c |&   a\\
    \end{array} $

\begin{theorem}\label{theoMOLS}
Suppose that
\begin{itemize} \item $M_2\subset M_1\subset M$ is a super
$MDS(2,5,p)$ code, \item $D$ is an $MDS(2,5,q)$ code, \item $E$ is
an $MDS(2,5,q_1-q)$ code on alphabet $A$,\item $F$ is an
$MDS(2,5,q_1)$ code with $4$-$A$-hole,\item  $G$ is an
$MDS(2,5,q_1)$ code with $5$-$A$-hole, where $|A|=q_1-q$.
\end{itemize}
 Then
 the set $C=E\cup (M_2\times_A G) \cup ((M_1\setminus M_2)\times_A F)\cup
((M\setminus M_1)\times D)$ is an $MDS(2,5,(p-1)q+q_1)$ code.
\end{theorem}
\begin{proof}
By the hypotheses of the theorem for any $y\in G$ there exist three
$i\in \{1,\dots,5\}$ such that $y_i\not \in A$. Since code distance
of $M_2$ equals $5$, for any $x,x'\in M_2$ all coordinates are
different. Consequently, if $(x,y)\neq (x',y')$ then $\pi_A(x,y)\neq
\pi_A(x',y')$ for $x,x'\in M_2$ and $y,y'\in G$. Therefore
$|M_2\times_A G|=|M_2\times G|=|M_2||G|$. By the same way we can
prove that $|M_2\times_A F|=|M_2||F|$ and $|M_1\times_A
F|=|M_1||F|$. Then it holds
$$|C|=|E|+|M_2||G|+(|M_1|-|M_2|)|F|+(|M|-|M_1|)|D|=$$
$$(q_1-q)^3+p(q_1^3-(q_1-q)^3)+(p^2-p)(q^3+3q^2(q_1-q))+
(p^3-p^2)q^3=(pq+q_1-q)^3.$$ The code distance of $X\times Y$ is the
minimum of the code distances of $X$ and $Y$. If elements of $Y$
contains not more than $k$ symbols from $A$ then $\rho(X\times_A
Y)\geq \min(\rho(X)-k,\rho(Y))$. Hence the interior distances of the
codes $E$, $M_2\times_A G$, $(M_1\setminus M_2)\times_A F)$ and
$(M\setminus M_1)\times D$ are not less than $3$ by the hypotheses
of the theorem. The distance between codes $(M\setminus M_1)\times
D$ and $E$ equals $5$. The distance between $(M\setminus M_1)\times
D$ and $(M_1\setminus M_2)\times_A F$ (or $M_2\times_A G$) is not
less than the distance between $(M\setminus M_1)\times D$ and
$(M_1\setminus M_2)\times F$ (or $M_2\times G$). This distance is
not less than the distance between $M\setminus M_1$ and
$M_1\setminus M_2$ (or $M_2$), i.\,e., it is not less than the code
distance of $M$.

We have that $U_2(E)\cap F=U_2(E)\cap G=\varnothing$ by the
definition of a code with $j$-$A$-hole. Thus the distance between
$E$ and  $(M_1\setminus M_2)\times_A F$ (or $M_2\times_A G$) is not
less than the distance between  $E$ and  $F$ or $G$, i.\,e., it is
not less than $3$.

The distance between  $M_1\setminus M_2$ and $M_2$ is equal to $4$.
Take $(x_0,x_1,x_2,x_3,x_4)$ from $M_1\setminus M_2$.
Each element of $(x_0,x_1,x_2,x_3,x_4)\times_A F$ contains not more
than one symbol from $A$. Consequently, the distance between
$(x_0,x_1,x_2,x_3,x_4)\times_A F$ and $M_2\times_A G$ is not less
than $4-1=3$.

By the Singleton bound (Proposition \ref{proNM5})  $C$ is an MDS
code.
\end{proof}

It is easy to see that the MDS code $C$  constructed by using the
theorem above contains subcodes of orders $q$ and $q_1$. These
subcodes  are $\overline{x}\times D$, where $\overline{x}\in
M\setminus M_1$, and $E\cup (\overline{x}\times_A G)$, where
$\overline{x}\in M_2$.

\begin{proposition}\label{proMOLS}
Let $k\leq p$ and  $i\in \{1,\dots,k\}$. Suppose
\begin{itemize} \item $M$ is an  $MDS(2,5,p)$ code that
contains $k$ disjoint $MDS(3,5,p)$ codes $C_i$,\item $D$ is an
$MDS(2,5,q)$ code,\item $F_i$  is an $MDS(2,5,q_1)$ code over
alphabet $Q_q\cup A_i$ with $4$-$A_i$-hole, where $|A_i|=q_1-q$,
$A_i\cap A_j=\varnothing$ if $i\neq j$.\end{itemize}
 Then
 the set $S=(\bigcup_{i=1}^k C_i\times_{A_i} F_i)\cup ((M\setminus \bigcup_{i=1}^k C_i)\times D)$
 is an $MDS(2,5,(p-k)q+kq_1)$ code with $4$-$B$-hole, where $B=\bigcup_{i=1}^k A_i$.
\end{proposition}
\begin{proof}
By the hypotheses of the proposition for any $y\in F_i$ there exists
$j_1,j_2\in 1,\dots,5$ such that $y_{j_1}, y_{j_2}\not \in A_i$.
Since code distance of $C_i$ equals $4$,  any $x,x'\in C_i$ coincide
in one coordinate at most. Consequently, if $(x,y)\neq (x',y')$ then
$\pi_{A_i}(x,y)\neq \pi_{A_i}(x',y')$ for $x,x'\in C_i$ and $y,y'\in
F_i$. Then $|C_i\times_{A_i} F_i|=|C_i||F_i|=|C_1||F_1|$.
  By direct calculation we obtain the following equalities:
$$|S|=k|C_1||F_1|+(|M|-k|C_1|)|D|=kp^2(q^3+3q^2(q_1-q))+(p^3-kp^2)q^3=$$ $$(pq)^3+3(pq)^2k(q_1-q).$$
The distance between $(M\setminus \bigcup_{i=1}^k C_i)\times D$ and
$\bigcup_{i=1}^k (C_i\times_{A_i} F_i)$ is not less than the
distance between $M\setminus (\bigcup_{i=1}^k C_i)$ and
$\bigcup_{i=1}^k C_i$. Since $A_i\cap A_j=\varnothing$ if $i\neq j$,
the distance between $ C_i\times_{A_i} F_i$ and $ C_j\times_{A_j}
F_j$ is not less than the distance between $C_i$ and $C_j$.  For
$i=1,\dots,k$ we have
$\rho(C_i\times_{A_i}F_i)\geq\min(\rho(C_i)-1,\rho(F_i))=3$. The
code distance of $(M\setminus \bigcup_{i=1}^k C_i)\times D$ are not
less than the minimum of  the code distances of $D$ and $M$.
Therefore, the code distance of $S$ equals $3$.

 Let us  prove that $S\cap U_3((\bigcup_{i=1}^k
 {A_i})^5)=\varnothing$. By definition of $4$-$A_i$-hole, each element of
 $F_i$ contains not more than one symbol from $A_i$. So, each element of
 $S$ contains not more than one symbol from $\bigcup_{i=1}^k
 {A_i}$.

Let us  prove that $U_3(S)=(Q_q\cup(\bigcup_{i=1}^k
 {A_i}))^5\backslash (\bigcup_{i=1}^k
 {A_i})^5$. Consider any vector $\overline{w}$ with $4$ or less
  coordinates from $\bigcup_{i=1}^k
 {A_i}$. Without lost of generality, we take
 $\overline{w}=((x_0,y_0),a_1, w_2, w_3, w_4)$, where
 $a_1\in A_1$. Since $F_1$ is an $MDS(2,5,q_1)$ code with $4$-$A_1$-hole,
  there is a vector
   $\overline{z}=(y_0,a_1,z_2,z_3,z_4)\in F_1$.
Since $C_1$ is an $MDS(3,5,p)$ code,  there exists a vector
$\overline{x}=(x_0,x_{1},x_{2},x_{3},x_{4})\in C_1$. Then the
distance between vectors $\overline{x}\times_{A_1}\overline{z}$ and
$\overline{w}$ is equal to $3$.
\end{proof}


\begin{lemma}\label{lemMOLS}
There exists an $MDS(2,5,6)$ code with $4$-$\{a,b\}$-hole.
\end{lemma}
The proof is by direct verification of the table below. \vskip7mm

  $
   \begin{array}{ccccccccccccc}
  a &b &2 &3 &0 & 1&\qquad           & 0 &1 &b &a &3 &2 \\
  1 &0 &b &a &2 &3&\qquad           &   a &b &0 &1 &2 &3 \\
  b &a &0 &1 &3 &2&\qquad           &      3 &2 &a &b &1 &0 \\
  3 &2 &a &b &1 &0&\qquad           &      b & a& 3& 2& 0&1 \\
  0 &3 &1 &2 &.  &. &\qquad           &      2 &0 &1 &3 &. &. \\
 2 &1 &3 &0 &. &. &\qquad           &   1 &3 &2 &0 &. &. \\
   \end{array} $

\vskip7mm

$ \begin{array}{ccccccccccccc}
        1& 0& a& b& 3& 2&\qquad           &  b& a& 1& 0& 2& 3\\
        b& a& 2& 3& 1& 0&\qquad           &  1& 0& a& b& 3& 2\\
        3& 2& b& a& 0& 1&\qquad           &  a& b& 2& 3& 0& 1\\
        a& b& 0& 1& 2& 3&\qquad           &  2& 3& b& a& 1& 0\\
        2& 1& 3& 0& .& .&\qquad           &  0& 2& 3& 1& .& .\\
        0& 3& 1& 2& .& .&\qquad           &  3& 1& 0& 2& .& .\\
  \end{array} $

\vskip7mm

$ \begin{array}{ccccccccccccc}
        2& 3& b& a& 1& 0&\qquad           &  a& b& 3& 2& 1& 0\\
        a& b& 1& 0& 3& 2&\qquad           &  3& 2& b& a& 0& 1\\
        0& 1& a& b& 2& 3&\qquad           &  b& a& 0& 1& 3& 2\\
        b& a& 3& 2& 0& 1&\qquad           &  0& 1& a& b& 2& 3\\
        3& 0& 2& 1& .& .&\qquad           &  1& 3& 2& 0& .& .\\
        1& 2& 0& 3& .& .&\qquad           &  2& 0& 1& 3& .& .\\
   \end{array} $

\vskip7mm

 $ \begin{array}{ccccccccccccc}
        b& a& 1& 0& 2& 3&\qquad           &  2& 3& a& b& 0& 1\\
        2& 3& a& b& 0& 1&\qquad           &  b& a& 2& 3& 1& 0\\
        a& b& 3& 2& 1& 0&\qquad           &  1& 0& b& a& 2& 3\\
        0& 1& b& a& 3& 2&\qquad           &  a& b& 1& 0& 3& 2\\
        1& 2& 0& 3& .& .&\qquad           &  3& 1& 0& 2& .& .\\
        3& 0& 2& 1& .& .&\qquad           &  0& 2& 3& 1& .& .\\
   \end{array} $

\vskip7mm

$ \begin{array}{ccccccccccccc}
        0& 2& 3& 1& .& .& \qquad   \       &  1& 2& 0& 3& .& .\\
        3& 1& 0& 2& .& .& \qquad          &  2& 1& 3& 0& .& .\\
        1& 3& 2& 0& .& .& \qquad          &  0& 3& 1& 2& .& .\\
        2& 0& 1& 3& .& .& \qquad          &  3& 0& 2& 1& .& .\\
        .& .& .& .& .& .& \qquad          &  .& .& .& .& .& .\\
        .& .& .& .& .& .& \qquad          &  .& .& .& .& .& .\\
   \end{array} $

\vskip7mm

$ \begin{array}{ccccccccccccc}
        3& 1& 0& 2& .& .&  \qquad  \       &  3& 0& 2& 1& .& .\\
        0& 2& 3& 1& .& .&  \qquad         &  0& 3& 1& 2& .& .\\
        2& 0& 1& 3& .& .&  \qquad         &  2& 1& 3& 0& .& .\\
        1& 3& 2& 0& .& .&  \qquad         &  1& 2& 0& 3& .& .\\
        .& .& .& .& .& .&  \qquad         &  .& .& .& .& .& .\\
        .& .& .& .& .& .&  \qquad         &  .& .& .& .& .& .\\
   \end{array} $

\begin{theorem}
If $q=16(6s\pm 1)+4$, then there exists an $MDS(2,5,q)$ code.
\end{theorem}
\begin{proof}
By Lemma \ref{lemMOLS} and Propositions  \ref{proNMbc}  and
\ref{proMOLS} ($p=q=4,k=2,q_1=6$), there exists an $MDS(2,5,20)$
code with $4$-$A$-hole, where $|A|=4$. By Theorem \ref{theoMOLS}
$(q_1=20,q=16,k=4)$ we can obtain an $MDS(2,5,16p+4)$ code if there
exists a super $MDS(2,5,p)$ code. Since any integer $p=6s\pm 1$ is
not divisible by $2$ and $3$, there exists a super $MDS(2,5,p)$ code
by Propositions \ref{proNMbb} and \ref{proNMb1}.
\end{proof}

By Proposition \ref{proEM} all $MDS(2,5,q)$ codes are equivalent to
pairs of orthogonal latin cubes of order $q$. If $6s-1=18i-1$ or
$6s-1=18i+5$, then pairs of orthogonal latin cubes of order
$q=16(6s-1)+4$ were not previously known because in these cases $q$
is divisible by $3$ but it is not divisible by $9$. Ten minimal new
obtained orders (not only of type $q=16(6s-1)+4$) are  $84, 132,
276, 372, 516, 564, 660, 852, 948, 1140$.

\section{ Connection between MDS codes and combinatorial designs}

A {\it Steiner system} with parameters $\tau, d, q$, $\tau\leq d$,
written $S(\tau,d,q)$, is  a set of $d$-element subsets of $Q_q$
(called {\it blocks}) with the property that each $\tau$-element
subset of $Q_q$ is contained in exactly one block.

\begin{theorem}  If  $D_2$ and $D_3$ are Steiner systems $S(2,5,q)$ and  $S(3,5,q)$ respectively
 and $D_2\subset D_3$, then there exits an
$MDS(2,5,q)$ code.\end{theorem}

\begin{proof}
 Consider a block $X=\{x_1,x_2,x_3,x_4,x_5\}\in    D_3\setminus D_2$. Define a
 set\\
  $M_X=\{(x_{\tau1},x_{\tau2},x_{\tau3},x_{\tau4},x_{\tau5})\ |\
  \tau \in Alt(5)\}$, where $Alt(5)$ is the alternating group.

By Proposition \ref{proNMb}  there exists an $MDS(2,5,5)$ code that
contains
 $(a,a,a,a,a)$ for all $a\in Q_5$.
 Suppose that $X=\{x_1,x_2,x_3,x_4,x_5\}\in  D_2$. Let us define an  $MDS(2,5,5)$
 code $ M_X$
 over the alphabet $X$   such that $M_X$ contains
 $(x_i,x_i,x_i,x_i,x_i)$ for $i=1,\dots,5$.
The intersections of pairs of such codes contain only elements of
type $(a,a,a,a,a)$ for $a\in
 Q_q$. 

 Let us to prove that $M=\bigcup_{X\in D_3}M_X$ is an $MDS(2,5,q)$
 code. The following holds:
 $$|M|=q+|D_2|(5^3-5)+(|D_3|-|D_2|)|Alt(5)|=$$ $$q+5\cdot24\frac{q(q-1)}{ 4\cdot
 5}+  {3\cdot 4\cdot
 5}\left(\frac{q(q-1)(q-2)}{ 3\cdot4\cdot
 5} - \frac{q(q-1)}{ 4\cdot
 5}\right)=q^3.$$

 Suppose that $X\in D_3\setminus D_2$, $Y\in D_3$ and $X\neq Y$. The
 distance between codes
 $M_X$ and $M_Y$ is not less than $3$ because $|X\cap Y|\leq 2$.
Suppose $X,Y\in  D_2$ and $X\neq Y$. Then $|X\cap Y|\leq 1$. If
$x\in M_X$ is not a constant vector, then it contains not more than
$2$ equal symbols. If $x\in M_X$ and $y\in M_Y$ are not constant
vectors, then $\rho(x,y)\geq 3$ by direct verification.

 If
 $x,y\in M_X$ and $X\in D_2$, then $\rho(x,y)\geq 3$ by the definition of
 $M_X$. Any non-constant permutation from $Alt(5)$ permutes $3$ or more elements. Therefore
 for $X\in D_3\setminus D_2$
   we obtain  that $\rho(x,y)\geq 3$ for any distinct $x,y\in M_X$.

   Thus we proved that the code distance of $M$ is at least $3$.
   So,
   $M$ is an
$MDS(2,5,q)$ code by the Singleton bound (Proposition \ref{proNM5}).
 \end{proof}

The natural divisibility conditions for the existence of Steiner
systems $S(2,5,n)$ and $S(3,5,n)$ simultaneously is that $n=5\
\mbox{ or}\ 41 \mod 60$. Steiner systems $S(3,5,41)$ are unknown.
Steiner systems $S(2,5,65)$ and $S(3,5,65)$  exist. Systems
$S(2,q+1,q^3+1)$ are unitals and  systems $S(3,q+1,q^3+1)$ are
spherical geometries if $q$ is a prime power ($q=4$ in this case).
But it is unknown whether the system $S(3,5,65)$ contains the system
$S(2,5,65)$.
 Keevash \cite{Keevash1} and Glock et al. \cite{Glock} proved that the
natural divisibility conditions are sufficient for existence of
Steiner system $S(t,k,n)$ (and inserted Steiner systems) apart from
a finite number of exceptional $n$ given fixed $t$ and $k$.
Therefore it is possible to use the theorem above for constructing
$MDS(2,5,q)$ codes if $q$ is large enough.

\section{Acknowledgments}

My sincere thanks are due to  D.S.Krotov who  programmed the
proposed method and calculated a series of new pairs orthogonal
latin cubes (see
https://ieee-dataport.org/open-access/graeco-latin-cubes).

\end{document}